\documentclass[12pt,a4paper]{article}

\usepackage{amssymb}
\usepackage{amsthm}
\usepackage{amsmath}
\usepackage[colorlinks=true,allcolors=black]{hyperref}
\usepackage{enumerate}

\newtheorem{theorem}{Theorem}[section]
\newtheorem{corollary}[theorem]{Corollary}
\newtheorem{lemma}[theorem]{Lemma}
\newtheorem{proposition}[theorem]{Proposition}

\begin{document}

\title{\bf On the minimum degree of power graphs of finite nilpotent groups}
\author{Ramesh Prasad Panda \and Kamal Lochan Patra \and Binod Kumar Sahoo}

\date{}

\maketitle

\begin{abstract}

The power graph $\mathcal{P}(G)$ of a group $G$ is the simple graph with vertex set $G$ and two vertices are adjacent whenever one of them is a positive power of the other. In this paper, for a finite noncyclic nilpotent group $G$, we study the minimum degree $\delta(\mathcal{P}(G))$ of $\mathcal{P}(G)$. Under some conditions involving the prime divisors of $|G|$ and the Sylow subgroups of $G$, we identify certain vertices associated with the generators of maximal cyclic subgroups of $G$ such that $\delta(\mathcal{P}(G))$ is equal to the degree of one of these vertices. As an application, we obtain $\delta(\mathcal{P}(G))$ for some classes of finite noncyclic abelian groups $G$.
\end{abstract}



\section{Introduction}

There are various graphs associated with groups which have been studied in the literature, e.g., Cayley graphs, commuting graphs, prime graphs, generating graphs.
The notion of directed power graph of a group was introduced by Kelarev and Quinn in \cite{kel-2000}.
The underlying undirected graph, simply termed as the `power graph', of a group was first considered by Chakrabarty et al. in \cite{CGS-2009}.
Several researchers have then studied both the directed and undirected power graphs of groups from different viewpoints, see \cite{AKC, KSCC} and the references therein.

The {\it power graph} of a group $G$, denoted by $\mathcal{P}(G)$, is the simple graph with vertex set $G$ and two vertices are adjacent whenever one of them is a positive power of the other. Cameron and Ghosh \cite{cameron} studied an isomorphism problem associated with a finite group and its power graph. Curtin and Pourgholi \cite{cur-2016} investigated an extremal property concerning the clique number of power graphs of finite groups. Kirkland et al. \cite{kirkland} derived explicit formulas for the
complexity of power graphs of various finite groups. Additionally, parameters of power graphs, such as chromatic number \cite{shitov}, spectra \cite{mehranian}, metric dimension \cite{FMW}, etc., have been studied recently. In \cite{MR3366577a}, Doostabadi et al. considered power graphs of finite nilpotent groups and obtained their diameter and number of components. Whereas, Chattopadhyay et al. \cite{CPS-2} investigated the vertex connectivity of the aforemetioned graphs.

In this paper, for a finite noncyclic nilpotent group $G$, we study the minimum degree of $\mathcal{P}(G)$ under certain conditions involving the prime divisors of the order $|G|$ of $G$ and the Sylow subgroups of $G$. Our objective will be to identify certain vertices of $\mathcal{P}(G)$ such that the minimum degree of $\mathcal{P}(G)$ is equal to the degree of one of these vertices.

\subsection{Graphs and groups}

Let $\Gamma$ be a simple graph with vertex set $V$. The {\it edge connectivity} of $\Gamma$, denoted by $\kappa'(\Gamma)$, is the minimum number of edges whose deletion gives a disconnected subgraph of $\Gamma$. For $v\in V$, we denote by $\deg(v)$ the degree of $v$ in $\Gamma$. The {\it minimum degree} of $\Gamma$, denoted by $\delta(\Gamma)$, is defined by $\delta(\Gamma)=\min \{\deg(v):v\in V\}$.
It is known that $\kappa'(\Gamma)\leq \delta(\Gamma)$ and that $\kappa'(\Gamma)=\delta(\Gamma)$ if the diameter of $\Gamma$ is at most $2$.

Let $G$ be a group with identity element $e$. The order of an element $x\in G$ is denoted by $o(x)$. Recall that $G$ is said to be {\it nilpotent} if its lower central series $G=G_1 \geq G_2 \geq G_3 \geq G_4 \geq \cdots $ terminates at $\{e\}$ after a finite number of steps, where $G_{i+1}=[G_i, G]$ for $i\geq 1$. If $G$ is a finite group, then $G$ being nilpotent is equivalent to any of the following statements:
\begin{enumerate}[\rm(a)]
\item Every Sylow subgroup of $G$ is normal (equivalently, there is a unique Sylow $p$-subgroup of $G$ for every prime divisor $p$ of $|G|$).
\item $G$ is the direct product of its Sylow subgroups.
\item For $x,y\in G$, $x$ and $y$ commute whenever $o(x)$ and $o(y)$ are relatively prime.
\end{enumerate}
Every abelian group is nilpotent. Also, every finite group of prime power order is nilpotent.

A cyclic subgroup $M$ of $G$ is called a {\it maximal cyclic subgroup} if it is not properly contained in any cyclic subgroup of $G$. We denote by $\langle x\rangle$ the cyclic subgroup of $G$ generated by an element $x\in G$ and by $\mathcal{M}(G)$ the collection of all maximal cyclic subgroups of $G$. If $G$ is finite, then every element of $G$ is contained in a maximal cyclic subgroup. This statement, however, need not hold in an infinite group: $(\mathbb{Q},+)$ is an example.

\subsection{Euler's totient function}

We denote by $\phi$ the Euler's totient function. Then $\phi(p^k)=p^{k-1}(p-1)=p^{k-1}\phi(p)$ for any prime $p$ and positive integer $k$, and $\underset{d\mid m}{\sum} \phi(d) = m$ for every positive integer $m$. Recall that $\phi$ is a multiplicative function, that is, $\phi(ab)=\phi(a)\phi(b)$ for any two positive integers $a,b$ which are relatively prime.

Throughout the paper, $n>1$ is a positive integer and $r$ is the number of distinct prime divisors of $n$. We write the prime power factorization of $n$ as
$$n=p_1^{\alpha_1}p_2^{\alpha_2} \cdots p_r^{\alpha_r},$$
where $\alpha_1, \alpha_2, \ldots, \alpha_r$ are positive integers and $p_1,p_2,\dots, p_r$ are prime numbers with $p_1 < p_2 < \cdots < p_r$. For a positive integer $m$, we define $[m]=\{1,2,\dots,m\}$.

We have $\phi(n)= \phi\left(p_1^{\alpha_1}\right)\phi\left(p_2^{\alpha_2}\right) \cdots \left(p_r^{\alpha_r}\right)$.
By \cite[Lemma 3.1]{CPS-2}, $(r+1)\phi\left(\underset{i\in [r]}\prod p_i\right) \geq \underset{i\in [r]}\prod p_i$ and equality holds if and only if $(r, p_1)=(1,2)$ or $(r, p_1,p_2)=(2,2,3)$. If $p_1 \geq r+1$, then it follows from the proof of Corollary 1.4 in \cite{CPS} that $2\phi\left(\underset{i\in [r]}\prod p_i\right) \geq \underset{i\in [r]}\prod p_i$ and equality holds if and only if $(r,p_1)=(1,2)$. Note that if $2\phi\left(\underset{i\in [r]}\prod p_i\right) \geq \underset{i\in [r]}\prod p_i$, then $2\phi\left(\underset{i\in I} \prod p_i\right)> \underset{i\in I} \prod p_i$ for any proper subset $I$ of $[r]$, where $\underset{i\in I} \prod p_i =1$ if $I=\emptyset$.

\subsection{Known results}

Let $G$ be a finite group. Note that two distinct vertices $x,y\in G$ are adjacent in $\mathcal{P}(G)$ if and only if $x\in\langle y\rangle$ or $y\in\langle x\rangle$. Observe that $\deg(x) = \deg(y)$ if $\langle x \rangle = \langle y \rangle$ and that $\deg(x)=o(x)-1$ if $\langle x\rangle\in \mathcal{M}(G)$. The identity element $e$ of $G$ is adjacent to all other vertices of $\mathcal{P}(G)$. So $\mathcal{P}(G)$ is of diameter at most 2 and hence $\kappa'(\mathcal{P}(G))=\delta(\mathcal{P}(G))$.  Thus it is enough to find the minimum degree of $\mathcal{P}(G)$ in order to determine the edge connectivity of $\mathcal{P}(G)$. By \cite[Theorem 2.12]{CGS-2009}, $\mathcal{P}(G)$ is a complete graph if and only if $G$ is cyclic group of prime power order.

Let $C_n$ denote the cyclic group of order $n$. If $r=1$, then $n$ is a prime power and so $\mathcal{P}(C_n)$ is a complete graph. In this case, $\delta(\mathcal{P}(C_n))=n-1=\deg(x)$ for every $x\in C_n$. For $r=2$, the minimum degree of $\mathcal{P}(C_n)$ was obtained in \cite[Theorem 4.6(i)]{PK-CA}. For $r\geq 3$, if $2\phi\left(\underset{i\in [r]} \prod p_i\right)\geq \underset{i\in [r]} \prod p_i$ or if $\phi(p_{j+1})\geq r\phi(p_j)$ for every $j\in [r-1]$, then at most $r-1$ vertices of $\mathcal{P}(C_n)$ are identified in \cite[Theorem 1.3]{PPS-JAA} such that $\delta(\mathcal{P}(C_n))$ is equal to the degree of one of these vertices. Further, when $r=3$ or $n$ is a product of distinct primes, two such vertices of $\mathcal{P}(C_n)$ are obtained in \cite[Theorems 1.2, 1.5]{PPS-JAA} without imposing any condition on the prime divisors of $n$.

If $G$ is a finite $p$-group for some prime $p$, then $\delta(\mathcal{P}(G)) = \deg(y)$ for any $y\in G$ such that $\langle y\rangle\in\mathcal{M}(G)$ is of minimum order.
This can be proved as follows: We show that $\deg(x)\geq \deg(y)$ for every $x\in G$. Consider $\langle z\rangle\in\mathcal{M}(G)$ containing $x$. Then $o(z)\geq o(y)$. Since $\langle z\rangle$ is a cyclic $p$-group, the induced subgraph $\mathcal{P}(\langle z\rangle)$ of $\mathcal{P}(G)$ is complete. Then $\deg(x)\geq o(z)-1\geq o(y)-1=\deg(y)$.

Recently, it is proved in \cite[Proposition 2.7]{PPS-new} that if $G$ is a noncyclic nilpotent group of order $n$ with $r=2$ in which the Sylow $p_1$-subgroup is noncyclic and the Sylow $p_2$-subgroup is cyclic, then $\delta(\mathcal{P}(G)) = \deg(y_1)$, where $\langle y_1\rangle$ is a maximal cyclic subgroup of the Sylow $p_1$-subgroup of $G$ of minimum order.

\subsection{Main results}

Let $G$ be a noncyclic nilpotent group of order $n$. We denote by $P_i$ the unique Sylow $p_i$-subgroup of $G$ for $i\in [r]$. Then $|P_i|=p_i^{\alpha_i}$ for $i\in [r]$. The following theorem is proved in Section \ref{sec-main-2}.

\begin{theorem}\label{main-2}
Suppose that $r\geq 2$ and that $2\phi\left(\underset{i\in [r]}\prod p_i\right) \geq \underset{i\in [r]}\prod p_i$. If $\delta(\mathcal{P}(G))=\deg(x)$ for some $x\in G$, then the following hold:
\begin{enumerate}[\rm(i)]
\item $x$ is contained in a unique element $M$ of $\mathcal{M}(G)$.
\item There exist $i\in [r]$ and a generator $y$ of $M$ such that $x = y^{p_i^{\lambda_i}}$,
where $|M\cap P_i|=p_i^{\gamma_i}$ and $\lambda_i \in\{0,1,\ldots,\gamma_i\}$ is the largest integer with the property that $M$ is the only element of $\mathcal{M}(G)$ containing $y^{p_i^{\lambda_i}}$.
\end{enumerate}
\end{theorem}

In Theorem \ref{main-2}(ii), observe that $\lambda_i=\gamma_i$ if and only if $P_i$ is cyclic if and only if $\gamma_i=\alpha_i$.
We prove the following theorem in Section \ref{sec-main-3-4} improving Theorem \ref{main-2} under the assumption that the center $Z(P_j)$ of $P_j$ is noncyclic whenever $P_j$ is noncyclic for $j\in [r]$.

\begin{theorem}\label{main-3}
Suppose that $r\geq 2$ and that $Z(P_j)$ is noncyclic for every $j\in J:=\{i\in [r]: P_i \text{ is noncyclic}\}$. Then the following hold:
\begin{enumerate}[\rm(i)]
\item If $J\neq [r]$ and $2\phi\left(\underset{i\in [r]} \prod p_i\right)\geq \underset{i\in [r]} \prod p_i$, then there exists $k\in [r]\setminus J$ such that $\delta(\mathcal{P}(G)) = \deg\left(y^{p_k^{\alpha_k}}\right)$, where $\langle y\rangle\in\mathcal{M}(G)$ is of minimum order.
\item If $J=[r]$ and $2\phi\left(\underset{i\in [r]\setminus\{l\}} \prod p_i\right)\geq \underset{i\in [r]\setminus\{l\}} \prod p_i$ for every $l\in [r]$, then $\delta(\mathcal{P}(G)) = \deg(y)$, where $\langle y\rangle\in\mathcal{M}(G)$ is of minimum order.
\end{enumerate}
\end{theorem}

Note that if $p_1 \geq r+1$, then $2\phi\left(\underset{i\in [r]} \prod p_i\right)\geq \underset{i\in [r]} \prod p_i$ and hence Theorems \ref{main-2} and \ref{main-3} hold good. For $r\geq 3$, each of the conditions in Theorem \ref{main-3} involving the prime divisors of $n$ implies that $p_1\geq 3$. We prove the following theorem in Section \ref{sec-main-3-4} without any condition on the prime divisors of $n$, thus improving Theorem \ref{main-3}(ii) when $r=3$.

\begin{theorem}\label{main-4}
Suppose that $r=3$ and that $Z(P_j)$ is noncyclic for every $j\in [r]$. Then $\delta(\mathcal{P}(G)) = \deg(y)$, where $\langle y\rangle\in\mathcal{M}(G)$ is of minimum order.
\end{theorem}

Now consider $r=2$. If both $Z(P_1)$ and $Z(P_2)$ are noncyclic, then $\delta(\mathcal{P}(G))$ is determined by Theorem \ref{main-3}(ii). If $P_1$ is noncyclic and $P_2$ is cyclic, then $\delta(\mathcal{P}(G))$ is determined by \cite[Proposition 2.7]{PPS-new}. We prove the following theorem in Section \ref{sec-main-5}.

\begin{theorem}\label{main-5}
Suppose that $r=2$ and that $P_1$ is cyclic and $Z(P_2)$ is noncyclic. Then $\delta(\mathcal{P}(G)) = \deg(y_2)$, where $\langle y_2\rangle\in\mathcal{M}(P_2)$ is of minimum order.
\end{theorem}

If a Sylow subgroup of a finite abelian group is noncyclic, then it automatically implies that the center of that Sylow subgroup is noncyclic. Since every abelian group is nilpotent, we have the following corollaries for $r\in\{2,3\}$.

\begin{corollary}\label{cor-abelian-1}
Let $G$ be a noncyclic abelian group of order $n$ with $r=2$. Then the following hold:
\begin{enumerate}[\rm(i)]
\item If both $P_1$ and $P_2$ are noncyclic, then $\delta(\mathcal{P}(G)) = \deg(y)$, where $\langle y\rangle\in\mathcal{M}(G)$ is of minimum order.
\item If $P_1$ is noncyclic and $P_2$ is cyclic, then $\delta(\mathcal{P}(G)) = \deg(y_1)$, where $\langle y_1\rangle\in\mathcal{M}(P_1)$ is of minimum order.
\item If $P_1$ is cyclic and $P_2$ is noncyclic, then $\delta(\mathcal{P}(G)) = \deg(y_2)$,  where $\langle y_2\rangle\in\mathcal{M}(P_2)$ is of minimum order.
\end{enumerate}
\end{corollary}

\begin{corollary}\label{cor-abelian-2}
Let $G$ be an abelian group of order $n$ with $r=3$ in which every Sylow subgroup is noncyclic. Then $\delta(\mathcal{P}(G)) = \deg(y)$, where $\langle y\rangle\in\mathcal{M}(G)$ is of minimum order.
\end{corollary}

\section{Comparison of vertex degrees}

Throughout this section, $G$ is a nilpotent group of order $n$ with $r\geq 2$. Then $G=P_1P_2\cdots P_r$, an internal direct product of the Sylow subgroups $P_1,P_2,\ldots, P_r$ of $G$. For a given $x \in G$, there exists a unique element in $P_i$, denoted by $x_i$, for $i\in [r]$ such that
$x = x_1x_2 \cdots x_r$.
Then $\langle x\rangle=\langle x_1\rangle \langle x_2\rangle\cdots \langle x_r\rangle$. Further, $\langle x\rangle\in\mathcal{M}(G)$ if and only if $\langle x_i\rangle\in\mathcal{M}(P_i)$ for every $i\in [r]$, see \cite[Lemma 2.11]{CPS-2}. Therefore, $\langle x\rangle\in\mathcal{M}(G)$ is of minimum order if and only if $\langle x_i\rangle\in\mathcal{M}(P_i)$ is of minimum order for every $i\in [r]$. For a given $M\in\mathcal{M}(G)$, we have $M_i:=M\cap P_i\in \mathcal{M}(P_i)$ for $i\in [r]$ and $M =  M_1 M_2 \cdots M_r$.

For $x\in G$, let $\tau_x$ denote the subset of $[r]$ consisting of those elements $j\in [r]$ for which the $j$-th component $x_j$ of $x$ is nonidentity, that is,
$\tau_x:=\{j\in [r]:\;x_j\neq e\}$.
Note that $\tau_x=[r]$ if $\langle x\rangle\in\mathcal{M}(G)$ and that $\tau_x$ is empty if and only if $x=e$. We have $\deg(e)=|G|-1=n-1\geq\deg(x)$ for every $x\in G\setminus\{e\}$.

The following two results were proved in \cite[Propositions 2.3, 2.4]{PPS-new}. The first one gives a lower bound for the degrees of vertices of $\mathcal{P}(G)$ and provides necessary and sufficient conditions to attain that bound. The second one determines certain vertices of minimum degree among the vertices contained in a Sylow subgroup of $G$.

\begin{proposition}[\cite{PPS-new}]\label{degprime}
Let $x \in G\setminus\{e\}$ and $M\in\mathcal{M}(G)$ containing $x$. If $|M|=p_{1}^{\gamma_{1}}p_{2}^{\gamma_{2}} \cdots p_{r}^{\gamma_{r}}$ and $o(x)=\underset{i\in \tau_x}\prod p_i^{\beta_i}$, where $1\leq \gamma_j\leq \alpha_j$ for $j\in [r]$ and $1\leq \beta_i\leq \gamma_i$ for $i\in \tau_x$, then
\begin{equation}\label{eqthm}
\deg(x) \geq  o(x) - \phi(o(x))
+ \left( \underset{j\in[r]\setminus \tau_x} \prod p_j^{\alpha_j} \right) \left(\underset{i\in \tau_x}\prod \left(p_{i}^{\gamma_{i}} - p_{i}^{\beta_i-1}\right)\right)  - 1.
\end{equation}
Further, equality holds in (\ref{eqthm}) if and only if $x$ belongs to exactly one maximal cyclic subgroup of $\underset{i\in \tau_x}\prod P_{i}$, namely $\underset{i\in \tau_x}\prod M_{i}$.
\end{proposition}

\begin{proposition}[\cite{PPS-new}]\label{mindeg_in_pi}
For $k\in [r]$, suppose that $\langle y_k \rangle \in \mathcal{M}(P_k)$ is of minimum order. Then for any $x_k\in P_k$, $\deg(y_k)\leq \deg(x_k)$ in $\mathcal{P}(G)$.
\end{proposition}

As an application of Proposition \ref{degprime}, we have the following.

\begin{corollary}\label{rem2}
Let $x,y \in G$ with $\tau_x=\tau_y$. If $\langle x \rangle$ and $\langle y \rangle$ are maximal cyclic subgroups  of $\underset{i\in \tau_x}\prod P_{i}$ of equal order, then $\deg(x) = \deg(y)$  in $\mathcal{P}(G)$.
\end{corollary}

In the following proposition, we obtain a comparison between the degree of a vertex generating a maximal cyclic subgroup of $G$ and that of some related vertex. 

\begin{proposition}\label{propcomp}
Let $\langle y \rangle \in \mathcal{M}(G)$ and $I$ be a nonempty proper subset of $[r]$. For $x:=\underset{i\in I}\prod y_i$, the following hold:
\begin{enumerate}[\rm(i)]
\item If $P_j$ is cyclic for every $j\in [r]\setminus I$, then  $\deg(y) > \deg(x)$.
\item If $P_k$ is noncyclic for some $k\in [r]\setminus I$ and $2\phi\left(\underset{i\in I}\prod p_i\right) \geq \underset{i\in I}\prod p_i$, then  $\deg(y) < \deg(x)$.
\end{enumerate}
\end{proposition}

\begin{proof}
We have $\tau_y=[r]$ and $\tau_x=I$. Put $o(y)= p_1^{\gamma_1}p_2^{\gamma_2} \cdots p_r^{\gamma_r}$, where $1\leq \gamma_i\leq \alpha_i$ for $i\in [r]$. Then $\deg(y)=o(y)-1=p_1^{\gamma_1}p_2^{\gamma_2} \cdots p_r^{\gamma_r} -1$ and $o(x)=\underset{i\in \tau_x}\prod p_i^{\gamma_i}$. Since $\langle x\rangle$ is a maximal cyclic subgroup of $\underset{i\in \tau_x}\prod P_{i}$, Proposition \ref{degprime} gives that
\begin{align*}
\deg(x)& =o(x)-\phi(o(x)) + \left( \underset{j\in[r]\setminus \tau_x} \prod p_j^{\alpha_j} \right) \left(\underset{i\in \tau_x}\prod \left(p_{i}^{\gamma_{i}} - p_{i}^{\gamma_i-1}\right)\right)  - 1\nonumber\\
 & = o(x)-\phi(o(x)) + \left( \underset{j\in[r]\setminus \tau_x} \prod p_j^{\alpha_j} \right) \phi(o(x))  - 1\nonumber \\
 & = o(x) + \phi(o(x))\left(\underset{j\in[r]\setminus \tau_x} \prod p_j^{\alpha_j} -1 \right) -1.
\end{align*}
Therefore,
\begin{align}\label{eqn-2}
\deg(x)-\deg(y)& =\phi(o(x))\left(\underset{j\in[r]\setminus \tau_x} \prod p_j^{\alpha_j} -1 \right) - o(x) \left( \underset{j\in [r]\setminus \tau_x} \prod p_j^{\gamma_j} -1 \right).
\end{align}

\noindent
(i) Since $P_j$ is cyclic for every $j\in [r]\setminus \tau_x$, we have $\gamma_j = \alpha_j$ for $j\in [r]\setminus \tau_x$ and it follows from (\ref{eqn-2}) that $\deg(x) < \deg(y)$ as $\phi(o(x))<o(x)$.\medskip

\noindent	
(ii) Since $P_k$ is noncyclic for some $k\in [r]\setminus \tau_x$, we have $\alpha_{k} > \gamma_k$ and so
\begin{align}\label{eqn-3}
\left( \underset{j\in [r]\setminus \tau_x} \prod p_j^{\alpha_j} -1 \right) - 2 \left(\underset{j\in [r]\setminus \tau_x} \prod p_j^{\gamma_j} -1 \right) &=  \left( \underset{j\in [r]\setminus \tau_x} \prod p_j^{\alpha_j}\right)  - 2 \left( \underset{j\in [r]\setminus \tau_x} \prod p_j^{\gamma_j}\right) +1 \nonumber \\
& \geq \left( \underset{j\in [r]\setminus (\tau_x\cup\{k\})} \prod p_j^{\gamma_j}\right) \left( p_k^{\alpha_k} - 2p_k^{\gamma_k}  \right)+1 > 0.
\end{align}
Since $2\phi\left(\underset{i\in \tau_x}\prod p_i\right) \geq \underset{i\in \tau_x}\prod p_i$ by the given hypothesis, we have
\begin{equation}\label{eqn-4}
2\phi(o(x))-o(x)=\left(\underset{i\in \tau_x}\prod p_i^{\gamma_i -1}\right) \left[2\phi\left(\underset{i\in \tau_x}\prod p_i\right)-\underset{i\in \tau_x}\prod p_i\right] \geq 0.
\end{equation}
Using (\ref{eqn-2}), (\ref{eqn-3}) and (\ref{eqn-4}), it follows that $ 2[\deg(x) - \deg(y)] >0$ and hence $\deg(x) > \deg(y)$.
\end{proof}

Note that $2\phi(p_j) \geq p_j$ for every $j\in [r]$. So we have the following corollary as an application of Proposition \ref{propcomp}(ii).

\begin{corollary}
For any $\langle y \rangle \in \mathcal{M}(G)$, the following hold:
\begin{enumerate}[\rm(i)]
\item If $P_k$ is noncyclic for some $k\in [r]$, then $\deg(y) < \deg(y_j)$ for every $j \in [r]\setminus \{k\}$.
\item If $P_k$ and $P_l$ are noncyclic for distinct $k,l\in [r]$, then $\deg(y) < \deg(y_i)$ for every $i\in [r]$.
\end{enumerate}
\end{corollary}

In the following proposition, under some conditions, we compare the degree of a vertex $x$ with that of some vertex $y$ such that $x\in\langle y\rangle$.

\begin{proposition}\label{proplessthan}
Let $x\in G$ be such that $|\tau_x|\geq 2$ and $x$ is contained in exactly one maximal cyclic subgroup of $\underset{i\in \tau_x}\prod P_{i}$. Suppose that, for some $k,l\in \tau_x$ with $k < l$, there exists $y\in G$ such that $y^{p_{k}}=x$ and $\langle y_{l} \rangle\notin \mathcal{M}(P_{l})$. If $2\phi\left(\underset{i\in \tau_x}\prod p_i\right) \geq \underset{i\in \tau_x}\prod p_i$, then $\deg(y) < \deg(x)$.
\end{proposition}

\begin{proof}
For $i\in \tau_x$, let $M_i$ be the unique element of $\mathcal{M}(P_i)$ containing $x_i$. Since $y^{p_{k}}=x$, it follows that $\tau_y=\tau_x$ and that $\underset{i\in \tau_x}\prod M_{i}$ is the unique maximal cyclic subgroup of $\underset{i\in \tau_x}\prod P_{i}$ containing $y$.

Put $|M_i|=p_i^{\gamma_i}$ and $o(y) = \underset{i\in \tau_x}\prod p_{i}^{\beta_{i}}$, where $1\leq \beta_i\leq \gamma_i$ for $i\in \tau_x$. Since $y^{p_k}=x$ and $p_k$ divides $o(y)$, we have $o(x) =o(y)/p_k$. Then $x_k\neq e$ implies that $\beta_k\geq 2$. By Proposition \ref{degprime}, we have
\begin{equation*}
\deg(x) =  o(x) - \phi(o(x))
+ \left( \underset{j\in[r]\setminus \tau_x} \prod p_j^{\alpha_j} \right) \left(\underset{i\in \tau_x\setminus\{k\}}\prod \left(p_{i}^{\gamma_{i}} - p_{i}^{\beta_i-1}\right)\right)\left(p_{k}^{\gamma_{k}} - p_{k}^{\beta_k-2}\right)  - 1
\end{equation*}
and
\begin{equation*}
\deg(y) =  o(y) - \phi(o(y))
+ \left( \underset{j\in[r]\setminus \tau_x} \prod p_j^{\alpha_j} \right) \left(\underset{i\in \tau_x}\prod \left(p_{i}^{\gamma_{i}} - p_{i}^{\beta_i-1}\right)\right)  - 1.
\end{equation*}
Writing $\Theta = \deg(x) - \deg(y)$, we have
\begin{align*}
\Theta & = \left( \underset{j\in[r]\setminus \tau_x} \prod p_j^{\alpha_j} \right) \left(\underset{i\in \tau_x\setminus\{k\}}\prod \left(p_{i}^{\gamma_{i}} - p_{i}^{\beta_i-1}\right)\right) \left(p_k^{\beta_k-1} - p_k^{\beta_k-2}\right)\\
& \quad - p_k^{\beta_k -2}\left(\underset{i\in \tau_x\setminus\{k\}}\prod p_i^{\beta_i -1}\right)\left[\underset{i\in \tau_x}\prod p_i - \phi\left(\underset{i\in \tau_x}\prod p_i\right) \right](p_k -1)\\
& = \phi\left(p^{\beta_k-1}_{n_k}\right) \left(\underset{i\in \tau_x\setminus\{k\}}\prod p_i^{\beta_i -1}\right)\times\\
&\quad \left[\left( \underset{j\in[r]\setminus \tau_x} \prod p_j^{\alpha_j} \right) \left( \underset{i\in \tau_x\setminus\{k\}}\prod \left(p_i^{\gamma_i-\beta_i+1} - 1\right) \right) + \phi\left(\underset{i\in \tau_x}\prod p_i\right) - \underset{i\in \tau_x}\prod p_i  \right]\\
& = \phi\left(p^{\beta_k-1}_{n_k}\right) \left(\underset{i\in \tau_x\setminus\{k\}}\prod p_i^{\beta_i -1}\right)\times \\
& \quad \left[\phi\left(\underset{i\in \tau_x\setminus\{k\}}\prod p_i\right)\left[\left( \underset{j\in[r]\setminus \tau_x} \prod p_j^{\alpha_j} \right) \left( \underset{i\in \tau_x\setminus\{k\}}\prod \left(p_i^{\gamma_i-\beta_i}+\cdots +p_i + 1\right) \right)+ \phi(p_k)\right]  - \underset{i\in \tau_x}\prod p_i \right] .
\end{align*}
Note that $\beta_l < \gamma_l$ as $\langle x_{l} \rangle$ is not a maximal cyclic subgroup of $P_{l}$. Since $k<l$, we have $p_k<p_l$ and so $p_l +1 > \phi(p_k)$. This gives
$$\left( \underset{j\in[r]\setminus \tau_x} \prod p_j^{\alpha_j} \right) \left( \underset{i\in \tau_x\setminus\{k\}}\prod \left(p_i^{\gamma_i-\beta_i}+\cdots +p_i + 1\right) \right) > \phi(p_{k}).$$
Therefore,
$$\Theta > \phi\left(p^{\beta_k-1}_{n_k}\right) \left(\underset{i\in \tau_x\setminus\{k\}}\prod p_i^{\beta_i -1}\right)\left[2\phi\left(\underset{i\in \tau_x}\prod p_i\right)- \underset{i\in \tau_x}\prod p_i\right]\geq 0,$$
where the last inequality holds using the given hypothesis that $2\phi\left(\underset{i\in \tau_x}\prod p_i\right) \geq \underset{i\in \tau_x}\prod p_i$.
\end{proof}

\section{Proof of Theorem \ref{main-2}}\label{sec-main-2}

In this section, $G$ is a noncyclic nilpotent group of order $n$ with $r\geq 2$. We start with the following lemma, which gives a sufficient condition for a vertex attaining the minimum degree in $\mathcal{P}(G)$ to be contained in a unique maximal cyclic subgroup of $G$.

\begin{lemma}\label{deg.comp1}
Let $x\in G\setminus\{e\}$ be such that $\deg(x)=\delta(\mathcal{P}(G))$. If $2\phi\left(\underset{i\in [r]}\prod p_i\right) \geq \underset{i\in [r]}\prod p_i$, then $x$ is contained in a unique element of $\mathcal{M}(G)$.
\end{lemma}

\begin{proof}
Suppose that $M$ and $M'$ are two distinct elements of $\mathcal{M}(G)$ containing $x$, where $M$ is of least possible order among all such maximal cyclic subgroups of $G$. Then $M_i$ is an element of $\mathcal{M}(P_i)$ of minimum order containing $x_i$ for each $i\in [r]$. So $|M_i|\leq |M'_i|$ and hence $\phi(|M_i|)\leq \phi(|M'_i|)$ for each $i\in [r]$. Then $\phi(|M'|)=\underset{i\in [r]}\prod \phi(|M'_i|)\geq \underset{i\in [r]}\prod \phi(|M_i|)=\phi(|M|)$.

Let $M=\langle y \rangle$ for some $y\in G$. We have $\deg(y)=o(y)-1=|M|-1$. Since $x$ is neither a generator of $M$ nor a generator of $M'$, it follows that
$$\deg(x) \geq o(x) - 1 + \phi(|M'|)+\phi(|M|)\geq o(x) - 1+2\phi(|M|)= o(x) - 1+2\phi(o(y)).$$
Then
\begin{align*}
\deg(x) - \deg(y)  > 2\phi(o(y)) - o(y) = \frac{o(y)}{p_1p_2\cdots p_r}\left[2\phi\left(\underset{i\in [r]}\prod p_i\right) - \underset{i\in [r]}\prod p_i\right] \geq 0.
\end{align*}
This gives $\deg(x) >\deg(y)\geq \delta(\mathcal{P}(G))$, a contradiction.
\end{proof}

The proof of the following lemma is straightforward.

\begin{lemma}\label{cor2}
Suppose that $\delta({\mathcal{P}(G)})=\deg(x)$ for some $x \in G\setminus\{e\}$ contained in a unique element of $\mathcal{M}(G)$. Then $P_k$ is cyclic for every $k\in [r]\setminus \tau_x$. In particular, $\tau_x=[r]$ if every Sylow subgroup of $G$ is noncyclic.
\end{lemma}

The following proposition gives a sufficient condition so that the order of a vertex attaining the minimum degree in $\mathcal{P}(G)$ is divisible by at least $r-1$ distinct primes.

\begin{proposition}\label{propimp}
Let $\delta({\mathcal{P}(G)})=\deg(x)$ for some $x \in G\setminus\{e\}$. Suppose that $x$ is contained in a unique element of $\mathcal{M}(G)$ and that $2\phi\left(\underset{i\in \tau_x}\prod p_i\right) \geq \underset{i\in \tau_x}\prod p_i$. Then $|[r]\setminus \tau_x|\leq 1$.
\end{proposition}

\begin{proof}
Let $M$ be the unique element of $\mathcal{M}(G)$ containing $x$. Put $|M| = \underset{i\in [r]}\prod p_i^{\gamma_i}$ and $o(x) = \underset{i\in \tau_x}\prod p_i^{\beta_i}$, where $1\leq \gamma_i\leq \alpha_i$ for $i\in [r]$ and $1\leq \beta_i\leq \gamma_i$ for $i\in \tau_x$.

On the contrary, suppose that $|[r]\setminus \tau_x|\geq 2$. Let $k, l \in [r]\setminus \tau_x$ with $k < l$. Consider $z_k \in M_k$ with $o(z_k) = p_k$. Then $x \in \langle xz_k \rangle$, as $(xz_k)^{p_k}=x^{p_k}z_k^{p_k}=x^{p_k}$ and $\langle x\rangle=\langle x^{p_k}\rangle$. Note that $M$ is the only element of $\mathcal{M}(G)$ containing $xz_k$ as the same is true for $x$. Thus, $\underset{i\in \tau_x}\prod M_{i}$ is the only maximal cyclic subgroup of $\underset{i\in \tau_x}\prod P_{i}$ containing $x$, and $\underset{i\in \tau_x\cup\{k\}}\prod M_{i}$ is the only maximal cyclic subgroup of $\underset{i\in \tau_x\cup\{k\}}\prod P_{i}$ containing $xz_k$. By Proposition \ref{degprime}, we have
$$\deg(x)=o(x) - \phi(o(x))
+ \left( \underset{j\in[r]\setminus \tau_x} \prod p_j^{\alpha_j} \right) \left(\underset{i\in \tau_x}\prod \left(p_{i}^{\gamma_{i}} - p_{i}^{\beta_i-1}\right)\right)  - 1$$
and
$$\deg(xz_k)=o(xz_k) - \phi(o(xz_k))
+ \left( \underset{j\in[r]\setminus (\tau_x\cup\{k\})} \prod p_j^{\alpha_j} \right)(p_k^{\gamma_k} -1) \left(\underset{i\in \tau_x}\prod \left(p_{i}^{\gamma_{i}} - p_{i}^{\beta_i-1}\right)\right)  - 1,$$
as $o(z_k)=p_k$. Writing $\Theta:=\deg(x) -  \deg(xz_k)$, we get
\begin{align}\label{eqn-5}
\Theta & = \phi(o(x)) [\phi(p_k)-1] - o(x) (p_k-1)\nonumber \\
& \qquad + \left( \underset{j\in[r]\setminus (\tau_x\cup\{k\})} \prod p_j^{\alpha_j} \right)  \left(\underset{i\in \tau_x}\prod \left(p_{i}^{\gamma_{i}} - p_{i}^{\beta_i-1}\right)\right) \left(p_{k}^{\alpha_{k}} - p_{k}^{\gamma_{k}} + 1 \right).
\end{align}
For $i\in \tau_x$, we have
\begin{align*}
p_{i}^{\gamma_{i}} - p_{i}^{\beta_{i}-1} =p_{i}^{\beta_{i}-1} (p_i -1) \left(p_{i}^{\gamma_{i}-\beta_{i}} + \cdots + p_i+1\right) =\phi\left(p_{i}^{\beta_{i}}\right) \left(p_{i}^{\gamma_{i}-\beta_{i}} + \cdots + p_i+1\right).
\end{align*}
So (\ref{eqn-5}) becomes
\begin{align*}
\Theta &= \phi(o(x)) \left[p_k-2+ \left( \underset{j\in[r]\setminus (\tau_x\cup\{k\})} \prod p_j^{\alpha_j} \right)  \left(\underset{i\in \tau_x}\prod \left(p_{i}^{\gamma_{i}-\beta_{i}} + \cdots + p_i+1\right)\right) \left(p_{k}^{\alpha_{k}} - p_{k}^{\gamma_{k}} + 1 \right)\right]\\
& \qquad - o(x) (p_k-1)\\
& \geq \phi(o(x)) \left[p_k-2+ p_l^{\alpha_l} \right] - o(x) (p_k-1).
\end{align*}
As $k <l$, we have $p_k < p_l$ and so $p_k-2+ p_l^{\alpha_l}> 2(p_k -1)$. Then
$$\Theta > (p_k-1)[2\phi(o(x)) - o(x)]=(p_k-1)\left(\underset{i\in \tau_x}\prod p_i^{\beta_i -1}\right)\left[2\phi\left(\underset{i\in \tau_x}\prod p_i\right) - \underset{i\in \tau_x}\prod p_i\right]\geq 0,$$
where the last inequality holds using the given hypothesis that $2\phi\left(\underset{i\in \tau_x}\prod p_i\right) \geq \underset{i\in \tau_x}\prod p_i$. We thus get $\deg(x) >\deg(xz_k)$, which contradicts the fact that $\delta({\mathcal{P}(G)}) = \deg(x)$.
\end{proof}

In the following proposition, we consider the case when the order of a vertex attaining the minimum degree in $\mathcal{P}(G)$ is divisible by exactly $r-1$ distinct primes.

\begin{proposition}\label{propimp-2}
Let $\delta({\mathcal{P}(G)})=\deg(x)$ for some $x \in G\setminus\{e\}$. Suppose that $x$ is contained in a unique element $M$ of $\mathcal{M}(G)$ and that $2\phi\left(\underset{i\in \tau_x}\prod p_i\right) \geq \underset{i\in \tau_x}\prod p_i$. If $[r]\setminus \tau_x=\{k\}$ for some $k\in [r]$, then
$x = y^{p_k^{\alpha_k}}$
for some generator $y$ of $M$.
\end{proposition}

\begin{proof}
Put $|M|=p_1^{\gamma_1} p_2^{\gamma_2} \cdots p_r^{\gamma_r}$, where $1\leq \gamma_i\leq\alpha_i$ for $i\in [r]$. Since $P_k$ is cyclic by Lemma \ref{cor2}, we have $\gamma_k=\alpha_k$. Let $z$ be a generator of $M$. Then $z^m=x$ for some integer $m$ with $1<m< |M|$. Since $x_k=e$, we have $z_k^m=e$ and so $m$ is divisible by $o(z_k)=p_k^{\alpha_k}$. We can write
$$m=a p_1^{\eta_1}\cdots p_{k-1}^{\eta_{k-1}}p_k^{\alpha_k}p_{k+1}^{\eta_{k+1}}\cdots p_r^{\eta_r},$$
where $0\leq \eta_i < \gamma_i$ for $i\in [r]\setminus\{k\}$ and $a$ is a positive integer relatively prime to $p_i$ for every $i\in [r]$.
Define $y:=z^a$ and $I:=\{i\in [r]\setminus\{k\}:\eta_i\geq 1\}\subseteq \tau_x$. Then $y$ is a generator of $M$ and $y^t=x$, where the integer $t$ is given by
$$t=p_k^{\alpha_k}\underset{i\in I}\prod p_i^{\eta_i}.$$
In order to prove the proposition, it is enough to show that $I$ is empty.

Suppose that $I$ is nonempty. First consider that $|I|\geq 2$. Let $j,l\in I\subseteq \tau_x$ with $j< l$. Taking $w:= y^{t/p_j}$, we get $w^{p_j}=x$. Further, $\langle w_l\rangle=\left\langle y_l^{t/p_j}\right\rangle $ is not a maximal cyclic subgroup of $P_l$ as $\eta_l\geq 1$. Since $2\phi\left(\underset{i\in \tau_x}\prod p_i\right) \geq \underset{i\in \tau_x}\prod p_i$, Proposition \ref{proplessthan} then implies that $\deg(w) < \deg(x)$, which contradicts the fact that $\deg(x)=\delta({\mathcal{P}(G)})$.

Now consider that $I=\{l\}$ for some $l\in [r]\setminus \{k\}$. Then $x = y^{p_k^{\alpha_k} p_l^{\eta_l}}$ with $\eta_l\geq 1$. In this case, we shall get a contradiction to the fact $\deg(x)=\delta({\mathcal{P}(G)})$ by showing that $\Theta_1:=\deg(x) - \deg\left(y^{p_l^{\eta_l}}\right)>0$ if $k <l$, and $\Theta_2:=\deg(x) - \deg\left(y^{p_k^{\alpha_k}}\right)>0$ if $k >l$. We have
\begin{align*}
o(x) & =o\left(y^{p_k^{\alpha_k} p_l^{\eta_l}}\right)=o(y)/p_k^{\alpha_k} p_l^{\eta_l}=p_l^{\gamma_l -\eta_l}\left(\underset{i\in [r]\setminus\{k,l\}}\prod p_i^{\gamma_i}\right),\\
o\left(y^{p_l^{\eta_l}}\right) & = o(y)/p_l^{\eta_l}=p_l^{\gamma_l -\eta_l}\left(\underset{i\in [r]\setminus\{k,l\}}\prod p_i^{\gamma_i}\right)p_k^{\alpha_k},\\
o\left(y^{p_k^{\alpha_k}}\right) & = o(y)/p_k^{\alpha_k}=p_l^{\gamma_l}\left(\underset{i\in [r]\setminus\{k,l\}}\prod p_i^{\gamma_i}\right).
\end{align*}
Note that $M$ is the only maximal cyclic subgroup of $G$ containing each of $y^{p_k^{\alpha_k}}$ and $y^{p_l^{\eta_l}}$ as the same is true for $x$. By Proposition \ref{degprime}, we have	
\begin{align*}
\deg(x)& =o(x) - \phi(o(x))
+ p_k^{\alpha_k} \left( p_l^{\gamma_l} - p_l^{\gamma_l - \eta_l-1}  \right) \left(\underset{i\in [r]\setminus\{k,l\}}\prod \left(p_{i}^{\gamma_{i}} - p_{i}^{\gamma_i-1}\right)\right)  - 1,\\
\deg\left(y^{p_l^{\eta_l}}\right) &= o\left(y^{p_l^{\eta_l}}\right)-\phi\left(o\left(y^{p_l^{\eta_l}}\right) \right) + \left( p_l^{\gamma_l} - p_l^{\gamma_l - \eta_l-1}  \right)\left(\underset{i\in [r]\setminus\{l\}}\prod \left(p_{i}^{\gamma_{i}} - p_{i}^{\gamma_i-1}\right)\right)-1, \\
\deg\left(y^{p_k^{\alpha_k}}\right) &= o\left(y^{p_k^{\alpha_k}}\right)-\phi\left( o\left(y^{p_k^{\alpha_k}}\right)\right) + p_k^{\alpha_k} \left(\underset{i\in [r]\setminus\{k\}}\prod \left(p_{i}^{\gamma_{i}} - p_{i}^{\gamma_i-1}\right)\right)-1,
\end{align*}
where $\gamma_k=\alpha_k$ in the second equality. First consider the case $k > l$. Then
\begin{align*}
\Theta_2 & = \phi\left( \underset{i\in [r]\setminus\{k,l\}} \prod p_i^{\gamma_i}\right) \left[ \phi\left(p_l^{\gamma_l}\right) - \phi\left(p_l^{\gamma_l - \eta_l}\right) \right] - \left(\underset{i\in [r]\setminus\{k,l\}}\prod p_i^{\gamma_i}\right) \left( p_l^{\gamma_l} - p_l^{\gamma_l - \eta_l}\right) \\
& \qquad \qquad\qquad \qquad\qquad  \qquad\qquad \qquad + p_{k}^{\alpha_{k}} \left( p_l^{\gamma_l- 1} - p_l^{\gamma_l - \eta_l-1}  \right) \phi \left( \underset{i\in [r]\setminus\{k,l\}}\prod  p_i^{\gamma_i} \right)  \\
&= \phi\left(\underset{i\in [r]\setminus\{k,l\}}\prod  p_i^{\gamma_i}\right) \left[\phi\left(p_l^{\gamma_l}\right) - \phi\left(p_l^{\gamma_l - \eta_l}\right)+ p_{k}^{\alpha_{k}} \left( p_l^{\gamma_l- 1} - p_l^{\gamma_l - \eta_l-1}  \right) \right]\\
& \qquad \qquad\qquad \qquad\qquad  \qquad\qquad  \qquad - \left(\underset{i\in [r]\setminus\{k,l\}}\prod p_i^{\gamma_i}\right) \left( p_l^{\gamma_l} - p_l^{\gamma_l - \eta_l}\right).
\end{align*}
Since $k > l$, we have $p_k > p_l\geq 2$ and so $p_k\geq p_l+1$. Using $1\leq \eta_l <\gamma_l$, we get
\begin{align*}
& \phi\left(p_l^{\gamma_l}\right) - \phi\left(p_l^{\gamma_l - \eta_l}\right)+ p_{k}^{\alpha_{k}} \left( p_l^{\gamma_l- 1} - p_l^{\gamma_l - \eta_l-1}  \right)  - 2 \left( p_l^{\gamma_l} - p_l^{\gamma_l - \eta_l}\right) \\
& \qquad = \left( p_l^{\gamma_l- 1} - p_l^{\gamma_l - \eta_l-1}  \right) \left[p_{k}^{\alpha_{k}} - (p_l+1) \right]\geq 0,
\end{align*}
giving $\phi\left(p_l^{\gamma_l}\right) - \phi\left(p_l^{\gamma_l - \eta_l}\right)+ p_{k}^{\alpha_{k}} \left( p_l^{\gamma_l- 1} - p_l^{\gamma_l - \eta_l-1}  \right)  \geq 2 \left( p_l^{\gamma_l} - p_l^{\gamma_l - \eta_l}\right)$. Thus
\begin{align*}
\Theta_2 \geq \left( p_l^{\gamma_l} - p_l^{\gamma_l - \eta_l}\right)\left[2\phi\left(\underset{i\in [r]\setminus\{k,l\}}\prod  p_i^{\gamma_i}\right) - \left(\underset{i\in [r]\setminus\{k,l\}}\prod p_i^{\gamma_i}\right)\right].
\end{align*}
Since $l\in \tau_x$, the hypothesis that $2\phi\left(\underset{i\in \tau_x}\prod p_i\right) \geq \underset{i\in \tau_x}\prod p_i$ implies $2\phi\left( \underset{i\in \tau_x\setminus\{l\}} \prod p_i\right) > \underset{i\in \tau_x\setminus\{l\}} \prod p_i$. It then follows that $\Theta_2 >0$ as $\eta_l <\gamma_l$ and $[r]\setminus\{k,l\}=\tau_x\setminus\{l\}$.
Next consider the case $k < l$. Then
\begin{align*}
\Theta_1 & = \phi\left( p_l^{\gamma_l - \eta_l}  \underset{i\in [r]\setminus\{k,l\}} \prod p_i^{\gamma_i}\right) \left(\phi\left( p_k^{\alpha_k}\right) - 1\right) - \left( p_l^{\gamma_l - \eta_l}  \underset{i\in [r]\setminus\{k,l\}} \prod p_i^{\gamma_i}\right) \left(p_k^{\alpha_k} - 1\right) \\
& \qquad \qquad\qquad \qquad\qquad  \qquad\qquad  \qquad\qquad + p_{k}^{\alpha_k-1} \left( p_l^{\gamma_l} - p_l^{\gamma_l - \eta_l-1}  \right) \phi \left( \underset{i\in [r]\setminus\{k,l\}} \prod p_i^{\gamma_i} \right)  \\
& =  p_l^{\gamma_l - \eta_l-1}  \left( \underset{i\in [r]\setminus\{k,l\}} \prod p_i^{\gamma_i-1} \right)\times \left[\phi\left( \underset{i\in [r]\setminus\{k\}} \prod p_i\right) \left(\phi\left( p_k^{\alpha_k}\right) - 1\right) - \left( \underset{i\in [r]\setminus\{k\}} \prod p_i \right) \left( p_k^{\alpha_k} - 1\right) \right.\\
& \qquad\qquad\qquad\qquad\qquad\qquad\qquad\qquad\qquad\qquad\quad \left. + p_{k}^{\alpha_k-1} \left( p_l^{\eta_l+1} - 1  \right) \phi \left(\underset{i\in [r]\setminus\{k,l\}} \prod p_i \right) \right]\\
& =  p_l^{\gamma_l - \eta_l-1}  \left( \underset{i\in [r]\setminus\{k,l\}} \prod p_i^{\gamma_i-1} \right)\times  \\
&  \quad \left[\phi\left( \underset{i\in [r]\setminus\{k\}} \prod p_i\right) \left[ \phi\left( p_k^{\alpha_k}\right) - 1 + p_{k}^{\alpha_k-1}\left( p_l^{\eta_l} + \dots + p_l + 1  \right) \right] - \left( \underset{i\in [r]\setminus\{k\}} \prod p_i \right) \left( p_k^{\alpha_k} - 1\right)  \right].
	\end{align*}
In the second equality above, we have used the fact that $\eta_l <\gamma_l$. Using $p_k < p_l$ (as $k< l$) and $\eta_l\geq 1$, we have
\begin{align*}
\phi\left( p_k^{\alpha_k}\right) - 1 + p_{k}^{\alpha_k-1}\left( p_l^{\eta_l} + \dots + p_l + 1  \right) - 2\left( p_k^{\alpha_k} - 1\right)&= p_k^{\alpha_k-1} \left( p_l^{\eta_l} + \dots + p_l \right)  -  p_k^{\alpha_k} + 1  \\
 & >   p_k^{\alpha_k-1} \left(p_l^{\eta_l} + \dots + p_l -  p_k \right)  \\
 & \geq   p_k^{\alpha_k-1} \left( p_l -p_k  \right) > 0.
 \end{align*}
 
 Thus,
\begin{align*}
\Theta_1  >  p_l^{\gamma_l - \eta_l-1}  \left( \underset{i\in [r]\setminus\{k,l\}} \prod p_i^{\gamma_i-1} \right)\times \left( p_k^{\alpha_k} - 1\right)\left[2\phi\left( \underset{i\in [r]\setminus\{k\}} \prod p_i\right)  - \left( \underset{i\in [r]\setminus\{k\}} \prod p_i \right)\right].
\end{align*}
Since $\tau_x=[r]\setminus\{k\}$ and $2\phi\left(\underset{i\in \tau_x}\prod p_i\right) \geq \underset{i\in \tau_x}\prod p_i$, it follows that $\Theta_1 >0$.
This completes the proof.
\end{proof}

For a nonempty proper subset $A$ of $G$, we denote by $\mathcal{P}(A)$ the induced subgraph of $\mathcal{P}(G)$ with vertex set $A$. In that case, for $x\in A$, the degree of $x$ in $\mathcal{P}(A)$ is denoted by $\deg_{\mathcal{P}(A)}(x)$.

Certain inequalities involving degree of vertices of $\mathcal{P}(C_n)$ were proved in \cite[Proposition 4.5]{PK-CA}. From the proof of \cite[Proposition 4.5(ii)]{PK-CA}, the following result with strict inequality can be obtained which we shall use in the proof Theorem \ref{main-2}.

\begin{lemma}[{\cite{PK-CA}}]\label{lem32}
Let $y$ be a generator of $C_n$, where $r\geq 2$. Then $\deg \left(y^{p_i^{\gamma}} \right) > \deg \left(y^{p_i^{\beta}} \right)$ in $\mathcal{P}(C_n)$ for every $i\in [r]$, where $0\leq \gamma < \beta \leq \alpha_i$.
\end{lemma}

\medskip
\begin{proof}[{\bf Proof of Theorem \ref{main-2}}]
Let $\delta(\mathcal{P}(G)) = \deg(x)$ for some $x\in G\setminus\{e\}$. The given hypothesis $2\phi\left(\underset{i\in [r]}\prod p_i\right) \geq \underset{i\in [r]}\prod p_i$ implies that $2\phi\left(\underset{i\in \tau_x}\prod p_i\right) \geq \underset{i\in \tau_x}\prod p_i$ and that there is a unique element $M$ of $\mathcal{M}(G)$ containing $x$ by Lemma \ref{deg.comp1}. Put $|M|=p_1^{\gamma_1} p_2^{\gamma_2} \cdots p_r^{\gamma_r}$, where $1\leq \gamma_i\leq\alpha_i$ for $i\in [r]$. \medskip

We first claim that $x = y^{p_j^{\lambda_j}}$ for some generator $y$ of $M$ and for some $j\in [r]$, where $0 \leq \lambda_j \leq \gamma_j$.	
If $M = \langle x \rangle$, then the claim follows by taking $y = x$ and any $j\in [r]$ with $\lambda_j = 0$. Assume that $M \neq \langle x \rangle$. By Proposition \ref{propimp}, we have either $[r]\setminus \tau_x=\{j\}$ for some $j\in [r]$, or $\tau_x=[r]$. In the first case, $\gamma_j=\alpha_j$ as $P_j$ is cyclic by Lemma \ref{cor2} and then the claim follows from Proposition \ref{propimp-2} with $\lambda_j=\gamma_j=\alpha_j$.

Next consider $\tau_x=[r]$. Let $z$ be a generator of $M$. Then $z^m=x$ for some integer $m$ with $1<m< |M|$. We can write
$m=a p_1^{\eta_1}p_{2}^{\eta_{2}}\cdots p_r^{\eta_r}$, where $0\leq \eta_i < \gamma_i$ for $i\in [r]$ and $a$ is a positive integer relatively prime to $p_i$ for every $i\in [r]$. We note that at least one $\eta_i$ is nonzero as $x$ is not a generator of $M$. Define $y:=z^a$ and $I:=\{i\in [r]:\eta_i>0\}$. Then $|I|\geq 1$, $y$ is a generator of $M$ and $y^t=x$, where $t$ is given by
$$t=\underset{i\in I}\prod p_i^{\eta_i}.$$
If we show that $|I|=1$, then the claim will follow. Suppose that $|I|\geq 2$. Let $k,l\in I$ with $k< l$. Taking $w:= y^{t/p_k}$, we get $w^{p_k}=y^{t}=x$. Further, $\langle w_l\rangle=\left\langle y_l^{t/p_k}\right\rangle \notin\mathcal{M}(P_l)$ as $\eta_l\geq 1$. Proposition \ref{proplessthan} then implies that $\deg(w) < \deg(x)=\delta({\mathcal{P}(G)})$, a contradiction. This completes the proof of our claim.\medskip

Now we determine the value of $\lambda_j$. Let $m$ be the largest integer in $\{0,1,\ldots,\gamma_j\}$ such that $M$ is the only maximal cyclic subgroup of $G$ containing $y^{p_j^{m}}$. We show that $\lambda_j=m$.
Since $M$ is the only maximal cyclic group of $G$ containing each of $y^{p_j^{\lambda_j}}=x$ and $y^{p_j^{m}}$, we have $m\geq \lambda_j$, $\deg_{\mathcal{P}(M)}\left(y^{p_j^{\lambda_j}} \right)=\deg\left(y^{p_j^{\lambda_j}} \right)$ and $\deg_{\mathcal{P}(M)}\left(y^{p_j^{m}} \right)=\deg\left(y^{p_j^{m}} \right)$. If $m > \lambda_j$, then $\deg_{\mathcal{P}(M)}\left(y^{p_j^{m}} \right) < \deg_{\mathcal{P}(M)}\left(y^{p_j^{\lambda_j}} \right)$ by Lemma \ref{lem32}. Thus
$$\deg\left(y^{p_j^{m}} \right)=\deg_{\mathcal{P}(M)}\left(y^{p_j^{m}} \right) < \deg_{\mathcal{P}(M)}\left(y^{p_j^{\lambda_j}} \right)=\deg\left(y^{p_j^{\lambda_j}} \right)=\deg(x)=\delta({\mathcal{P}(G)}),$$
a contradiction. This completes the proof.
\end{proof}

\section{Proof of Theorems \ref{main-3} and \ref{main-4}}\label{sec-main-3-4}

In this section, $G$ is a noncyclic nilpotent group of order $n$ with $r\geq 2$. For two distinct vertices $u,v\in G$, we write $u \sim v$ if they are adjacent in $\mathcal{P}(G)$.

\begin{proposition}\label{prop.abelian}
Let $x\in G\setminus\{e\}$. Suppose that the center $Z(P_k)$ of $P_k$ is noncyclic for some $k\in \tau_x$ and that one of the following two conditions holds:
\begin{enumerate}[\rm(i)]
\item $r=3$,
\item $2\phi \left(\underset{i\in \tau_x\setminus\{k\}}\prod p_i \right) \geq \underset{i\in \tau_x\setminus\{k\}}\prod p_i$.
\end{enumerate}
Then $\deg(x) \leq \deg\left(x^{p_k}\right)$.
\end{proposition}

\begin{proof}
Let $y=x^{p_k}$, $H=\langle x \rangle$ and $K=\langle y \rangle$. Note that if $z \sim x$ for some $z\in G\setminus H$, then $z \sim y$. If $|\tau_x|=1$, then $x$ is of prime power order and so the induced subgraph $\mathcal{P}(H)$ of $\mathcal{P}(G)$ is complete. This implies that $\deg_{\mathcal{P}(H)}(x) = \deg_{\mathcal{P}(H)}(y)$ and it follows that $\deg(x) \leq \deg(y)$. So assume that $|\tau_x|\geq 2$. Define the following three subsets of $G$:
\begin{align*}
A & = \{z \in G : z \sim x, z \not\sim y\},\\
B & = \{z \in G : z \sim y, z \not\sim x\},\\
C & = \{z \in G : z \sim x, z \sim y\}.
\end{align*}
Then $A,B,C$ are pairwise disjoint with $y\in A$ and $x\in B$. We have $\deg(x)= |A|+|C|$ and $\deg(y) = |B|+|C|$. In order to prove the proposition, it is enough to show that $|A| \leq |B|$.
	
Let $S$ be the set of generators of $H$. Then $A=(H\setminus (S\cup K))\cup \{y\}$. Since $S\cap K$ is empty, we have $|A| = |H| - \phi(|H|) - |K| +1 =o(x)-\phi(o(x))- o(y) +1$, where $o(y) = o(x)/p_k$. Let $o(x) = \underset{i\in \tau_x}\prod p_{i}^{\beta_{i}}$, where $1\leq \beta_i\leq \alpha_i$ for $i\in \tau_x$. Then
\begin{align*}
|A| & = \underset{i\in \tau_x}\prod p_{i}^{\beta_{i}-1}\left[ \underset{i\in \tau_x}\prod p_{i} -\phi\left(\underset{i\in \tau_x}\prod p_{i}\right) - \underset{i\in \tau_x\setminus\{k\}}\prod p_{i} \right] +1\\
	& = \phi(p_k)\left(\underset{i\in \tau_x}\prod p_{i}^{\beta_{i}-1}\right)\left[ \underset{i\in \tau_x\setminus \{k\}}\prod p_{i} -\phi\left(\underset{i\in \tau_x\setminus\{k\}}\prod p_{i}\right) \right] +1.
\end{align*}
	
Since $Z(P_k)$ is noncyclic, there exists an element $u\in Z(P_{k})\setminus H$ of order $p_{k}$.
Take $z: = ux$. Then $\tau_z=\tau_x$, $z\notin H=\langle x\rangle$ and $o(z)=o(x)$. This implies that $x\notin \langle z\rangle$ and so $z \nsim x$. But $z \sim y$  as $z^{p_{k}} =x^{p_k}= y$. Thus $z \in B$.
Note that $B$ is contained in $(G\setminus H)\cup\{x\}$ and that if $v \in G \setminus H$ such that $\underset{i\in \tau_x}\prod v_i$ is a generator for $\langle z\rangle$, then $v\in B$. Since $o(z)=o(x)$, we get
\begin{align*}
|B| \geq  \phi\left(\underset{i\in \tau_x} \prod p_{i}^{\beta_{i}} \right) \left(\underset{i\in [r]\setminus \tau_x} \prod p_{i}^{\alpha_{i}}\right) + 1
\geq \left(\underset{i\in \tau_x} \prod p_{i}^{\beta_{i}-1} \right)\phi\left(\underset{i\in \tau_x} \prod p_{i} \right) \left( \underset{i\in [r]\setminus \tau_x} \prod p_{i}\right) + 1, 	
\end{align*}
where $\underset{i\in [r]\setminus \tau_x} \prod p_{i}=1$ if $\tau_x=[r]$. Then
\begin{align}\label{eqn-use}
|B| - |A| & \geq \phi(p_k)\left(\underset{i\in \tau_x}\prod p_{i}^{\beta_{i}-1}\right) \left[\phi\left(\underset{i\in \tau_x\setminus\{k\}}\prod p_{i}\right)\left( 1+ \underset{i\in [r]\setminus \tau_x} \prod p_{i}\right) - \underset{i\in \tau_x\setminus \{k\}}\prod p_{i} \right].
\end{align}
Since $1+ \underset{i\in [r]\setminus \tau_x} \prod p_{i}\geq 2$, the inequality (\ref{eqn-use}) implies that $|B| - |A|\geq 0$ if condition (ii) holds.
	
Now, suppose that condition (i) holds, that is, $r=3$. Then $|\tau_x|\in\{2,3\}$. If $|\tau_x|=2$, then the condition (ii) holds and so we get $|A| \leq |B|$. Assume that $|\tau_x|=3=r$. If $p_{1} \geq 3$ or $(p_1, k) =(2,1)$, then the condition (ii) holds and so $|A| \leq |B|$.
	
So assume that $p_1 =2$ and $k> 1$. Then $p_{k} \geq 3$ and $o(u^2)=o(u)=p_k$. Take $w:= u^2x$. Then $w\notin H$ and $o(w)=o(x)$. The same argument used for $z$ in the above gives that $w \in B$.
We have $o(w)=o(z)$, but $\langle w \rangle \neq \langle z \rangle$. The latter follows from the fact that $\langle w_k\rangle =\langle u^2x_k\rangle\neq \langle ux_k\rangle=\langle z_k\rangle$ as $\langle u\rangle \cap \langle x_k\rangle =\{e\}$. Thus
$$|B| \geq  \phi(o(z)) +  \phi(o(w)) +1 = 2\phi(o(x)) +1 =2 \phi\left(p_{1}^{\beta_{1}} p_{2}^{\beta_{2}} p_{3}^{\beta_{3}}\right) +1.$$
Since $|A|=\phi(p_k)\left(p_{1}^{\beta_{1}-1} p_{2}^{\beta_{2}-1} p_{3}^{\beta_{3}-1}\right)\left[ \frac{p_1p_2p_3}{p_k} -\phi\left(\frac{p_1p_2p_3}{p_k} \right) \right] +1$, we get
\begin{align*}
|B| - |A|  \geq p_{1}^{\beta_{1}-1} p_{2}^{\beta_{2}-1} p_{3}^{\beta_{3}-1} \phi(p_{k})  \left[3\phi\left(\frac{p_1p_2p_3}{p_k} \right) - \frac{p_1p_2p_3}{p_k} \right] \geq 0.
\end{align*}
Hence $|A| \leq |B|$. This completes the proof.
\end{proof}

\medskip
\begin{proof}[{\bf Proof of Theorems \ref{main-3} and \ref{main-4}}]
Observe that if the condition $2\phi\left(\underset{i\in [r]}\prod p_i\right) \geq \underset{i\in [r]}\prod p_i$ holds, then $2\phi\left(\underset{i\in [r]\setminus\{l\}}\prod p_i\right) > \underset{i\in [r]\setminus\{l\}}\prod p_i$ for every $l\in [r]$ and hence $2\phi\left(\underset{i\in \tau_x\setminus\{l\}}\prod p_i\right) > \underset{i\in \tau_x\setminus\{l\}}\prod p_i$ for every $x\in G$ and every $l\in \tau_x$.
Note that $J\neq\emptyset$ as $G$ is noncyclic. By the given hypothesis, $Z(P_j)$ is noncyclic if and only if $j\in J$.

Suppose that $r=3$ or $2\phi\left(\underset{i\in [r]\setminus\{l\}} \prod p_i\right)\geq \underset{i\in [r]\setminus\{l\}} \prod p_i$ for every $l\in [r]$. Consider two subsets $A$ and $B$ of $G$ defined by
\begin{align*}
A & :=\{u\in G:\deg(u)=\delta(\mathcal{P}(G)) \text{ and } J\subseteq \tau_u\},\\
B & :=\{u\in A: \langle u_j\rangle\in\mathcal{M}(P_j) \text{ for every }j\in J\}.
\end{align*}
We assert that $A$ and $B$ are nonempty.

\underline{$A\neq \emptyset$}: Choose an element $w\in G$ such that $\deg(w)=\delta(\mathcal{P}(G))$ and $|\tau_w\cap J|$ is maximum. We claim that $J\subseteq \tau_w$. Suppose that $k\notin \tau_w$ for some $k\in J$. Let $z_k\in P_k$ be an element of order $p_k$. Define $x:=wz_k\in G\setminus\{e\}$. Then $x^{p_k}=(wz_k)^{p_k}=w^{p_k}$. Since $o(w)$ and $p_k$ are relatively prime, we have $\langle w\rangle =\langle w^{p_k}\rangle$ and so $\deg(w)=\deg(w^{p_k})=\deg(x^{p_k})$. Since $Z(P_k)$ is noncyclic, we apply Proposition \ref{prop.abelian} to get $\deg(x)\leq \deg(x^{p_k})=\deg(w)=\delta(\mathcal{P}(G))$. This implies $\deg(x)=\delta(\mathcal{P}(G))$. Since $|\tau_x\cap J|=|\tau_w\cap J|+1$, we get a contradiction to the maximality of $|\tau_w\cap J|$.\medskip

\underline{$B\neq \emptyset$}: Choose an element $v\in A$ such that the order of $\underset{j\in J}\prod v_j$ is maximum. We claim that $\langle v_j\rangle\in\mathcal{M}(P_j)$ for every $j\in J$. Suppose that this is not the case for some $k\in J$. Then there exists $z_k\in P_k$ such that $z_k ^{p_k}=v_k$. Define $x\in G$ such that
$$x_i=
\begin{cases}
v_i &\text{if $i\neq k$}\\
z_k &\text{if $i=k$}
\end{cases}.$$
Then $J\subseteq \tau_x$. For $i\in [r]\setminus \{k\}$, we have $\langle v_i\rangle =\langle v_i^{p_k}\rangle$ as $o(v_i)$ and $p_k$ are relatively prime. It follows that $\langle v\rangle=\langle x^{p_k}\rangle$ and so $\deg(v)=\deg(x^{p_k})$. Since $Z(P_k)$ is noncyclic, apply Proposition \ref{prop.abelian} to get $\deg(x)\leq \deg(x^{p_k})=\deg(v)=\delta(\mathcal{P}(G))$. This implies $\deg(x)=\delta(\mathcal{P}(G))$ and so $x\in A$. Since $o(z_k)>o(v_k)$, the order of the element $\underset{j\in J}\prod x_j$ is greater than that of $\underset{j\in J}\prod v_j$. This contradicts to the maximality of the order of $\underset{j\in J}\prod v_j$.\medskip

Now let $z\in B$. Then $\langle z_j\rangle\in\mathcal{M}(P_j)$ for every $j\in J$. Since $P_i$ is cyclic for every $i\in [r]\setminus J$, it follows that $z$ is contained in a unique element $M$ of $\mathcal{M}(G)$.\medskip

First consider $J=[r]$. Then $\langle z\rangle=M$ and so $\delta(\mathcal{P}(G))=\deg(z)=o(z)-1$. Since $\deg(x)=o(x)-1$ for every $\langle x\rangle\in \mathcal{M}(G)$, it follows that $\langle z\rangle$ must be of minimum order among all maximal cyclic subgroups of $G$. This proves Theorems \ref{main-3}(ii) and \ref{main-4}.\medskip

Next we prove Theorem \ref{main-3}(i). Accordingly, we consider $J\neq [r]$ and assume that $2\phi\left(\underset{i\in [r]}\prod p_i\right) \geq \underset{i\in [r]}\prod p_i$. Since $\delta(\mathcal{P}(G))=\deg(z)$, by Theorem \ref{main-2}(ii), there exists $k\in [r]$ and a generator $y$ of $M$ such that
$$z = y^{p_k^{\lambda_k}},$$
where $o(y_k)=p_k^{\gamma_k}$ with $1\leq \gamma_k\leq\alpha_k$ and $\lambda_k \in\{0,1,\ldots,\gamma_k\}$ is the largest integer with the property that $M$ is the only maximal cyclic subgroup of $G$ containing $y^{p_k^{\lambda_k}}$.

Suppose that $\lambda_k=0$. Then $z=y$ generates $M\in \mathcal{M}(G)$. Define $x:=\underset{j\in J} \prod z_j \in G$. Since $P_i$ cyclic for every $i\in [r]\setminus J$, Proposition \ref{propcomp}(i) implies that $\delta(\mathcal{P}(G))=\deg(z)>\deg(x)$, a contradiction. Thus $\lambda_k\geq 1$. This implies that $z_k=y_k^{p_k^{\lambda_k}}$ does not generate a maximal cyclic subgroup of $P_k$. So $k\notin J$, that is, $k\in [r]\setminus J$. Since $P_k$ is cyclic, it then follows from the definition of $\lambda_k$ that $\lambda_k=\gamma_k=\alpha_k$. Thus,
$z = y^{p_k^{\alpha_k}}$.
We note that $\deg\left(y^{p_k^{\alpha_k}}\right)$ is independent of the generator $y$ of $M$ by Corollary \ref{rem2}.

It remains to prove that $\langle y\rangle \in \mathcal{M}(G)$ is of minimum order. Consider $\langle w\rangle \in \mathcal{M}(G)$ of minimum order. Then $o(y)\geq o(w)$. Define $I:=[r]\setminus J\neq\emptyset$. Put $o(y)=\left(\underset{i\in I}\prod p_i^{\alpha_i}\right)\left(\underset{j\in J}\prod p_j^{\gamma_j}\right)$ and $o(w)=\left(\underset{i\in I}\prod p_i^{\alpha_i}\right)\left(\underset{j\in J}\prod p_j^{\delta_j}\right)$, where $1\leq \delta_j\leq \gamma_j <\alpha_j$ for every $j\in J$. Then $o\left(y^{p_k^{\alpha_k}}\right)=\left(\underset{i\in I\setminus\{k\}}\prod p_i^{\alpha_i}\right)\left(\underset{j\in J}\prod p_j^{\gamma_j}\right)$ and $o\left(w^{p_k^{\alpha_k}}\right)=\left(\underset{i\in I\setminus\{k\}}\prod p_i^{\alpha_i}\right)\left(\underset{j\in J}\prod p_j^{\delta_j}\right)$. We have
\begin{align*}
o\left(y^{p_k^{\alpha_k}}\right)-\phi\left(o\left(y^{p_k^{\alpha_k}}\right)\right) & =\left(\underset{i\in I\setminus\{k\}}\prod p_i^{\alpha_i -1}\right)\left(\underset{j\in J}\prod p_j^{\gamma_j -1}\right)\left[\underset{l\in [r]\setminus\{k\}}\prod p_l -\phi\left(\underset{l\in [r]\setminus\{k\}}\prod p_l\right)\right]\\
 & \geq \left(\underset{i\in I\setminus\{k\}}\prod p_i^{\alpha_i -1}\right)\left(\underset{j\in J}\prod p_j^{\delta_j -1}\right)\left[\underset{l\in [r]\setminus\{k\}}\prod p_l -\phi\left(\underset{l\in [r]\setminus\{k\}}\prod p_l\right)\right]\\
 & = o\left(w^{p_k^{\alpha_k}}\right)-\phi\left(o\left(w^{p_k^{\alpha_k}}\right)\right).
\end{align*}
Then by Proposition \ref{degprime}, we have	
\begin{align*}
\deg\left(y^{p_k^{\alpha_k}}\right) &= o\left(y^{p_k^{\alpha_k}}\right)-\phi\left(o\left(y^{p_k^{\alpha_k}}\right) \right) + p_k^{\alpha_k}\left(\underset{i\in I\setminus\{k\}}\prod \phi\left(p_i^{\alpha_i} \right)\right) \left( \underset{j\in J}\prod \phi\left(p_j^{\gamma_j} \right)\right) -1 \\
 &\geq  o\left(w^{p_k^{\alpha_k}}\right)-\phi\left(o\left(w^{p_k^{\alpha_k}}\right) \right) + p_k^{\alpha_k}\left(\underset{i\in I\setminus\{k\}}\prod \phi\left(p_i^{\alpha_i} \right)\right) \left( \underset{j\in J}\prod \phi\left(p_j^{\delta_j} \right)\right) -1\\
 & = \deg\left(w^{p_k^{\alpha_k}}\right).
\end{align*}
Since $\delta(\mathcal{P}(G))=\deg\left(y^{p_k^{\alpha_k}}\right)$, we must have $\deg\left(w^{p_k^{\alpha_k}}\right)=\deg\left(y^{p_k^{\alpha_k}}\right)$. This happens if only if $\delta_j =\gamma_j$ for every $j\in J$ and it follows that $o(y)=o(w)$. Thus the maximal cyclic subgroup $\langle y\rangle =M$ of $G$ is of minimum order.
\end{proof}

\section{Proof of Theorem \ref{main-5}} \label{sec-main-5}

Suppose that $G$ is a noncyclic nilpotent group of order $p_1^{\alpha_1}p_2^{\alpha_2}$ in which $P_1$ is cyclic and $Z(P_2)$ is noncyclic.
Let $y\in G$ be such that $M:=\langle y\rangle\in\mathcal{M}(G)$ is of minimum order. Put $|M|=o(y) = p_1^{\alpha_1}p_2^{\gamma_2}$, where $1\leq \gamma_2 <\alpha_2$.
Let $x\in G\setminus\{e\}$ be arbitrary. We claim that $\deg(x)\geq \deg(y_2)$. Applying Proposition \ref{propcomp}, we have
\begin{align}\label{eqn-6}
\deg(y_1)>\deg(y)>\deg(y_2).
\end{align}
If $|\tau_x|=1$, then $\deg(x) \geq \min \{\deg(y_1), \deg(y_2)\}=\deg(y_2)$ using (\ref{eqn-6}) and Proposition \ref{mindeg_in_pi}.

So assume that $|\tau_x|=2$. Consider $\langle z\rangle \in\mathcal{M}(G)$ containing $x$. Proposition \ref{degprime} gives that $\deg(y_2) = \dfrac{o(y_2)}{p_2} + p_1^{\alpha_1}\phi(o(y_2)) - 1$ and $\deg(z_2) = \dfrac{o(z_2)}{p_2} + p_1^{\alpha_1}\phi(o(z_2)) - 1$. Since $o(z_2)\geq o(y_2)$, it follows that $\deg(z_2) \geq \deg(y_2)$. We show that $\deg(x)\geq \deg(z_2)$.

Put $o(x)=p_1^{\beta_1}p_2^{\beta_2}$ and $o(z)=p_1^{\alpha_1}p_2^{\eta_2}$ for some positive integers $\beta_1,\beta_2,\eta_2$ with $\beta_2\leq \eta_2<\alpha_2$. Suppose that $\langle x_2\rangle = \langle z_2\rangle$. Then $\beta_2=\eta_2$. By Proposition \ref{degprime}, we have
\begin{align*}
\deg(x) & =  p_1^{\beta_1}p_2^{\eta_2} -  \phi\left(p_1^{\beta_1}p_2^{\eta_2}\right)+ \left(p_1^{\alpha_1}-p_1^{\beta_1-1}\right)\left(p_2^{\eta_2}-p_2^{\eta_2-1}\right)-1\\
 & = p_1^{\beta_1}p_2^{\eta_2-1} + p_1^{\alpha_1}p_2^{\eta_2} -  p_1^{\alpha_1}p_2^{\eta_2-1} -1
\end{align*}
and
\begin{align*}
\deg(z_2) = p_2^{\eta_2-1} + p_1^{\alpha_1}p_2^{\eta_2}-p_1^{\alpha_1}p_2^{\eta_2 -1}-1.
\end{align*}
Then $\deg(x) - \deg(z_2) = p_1^{\beta_1}p_2^{\eta_2-1} - p_2^{\eta_2-1}>0$ gives $\deg(x) > \deg(z_2)$. In particular, $\deg(x_1z_2)>\deg(z_2)$.

Now suppose that $x_2$ is not a generator of $\langle z_2\rangle$. Then $l:=\eta_2 -\beta_2 >0$. There exist elements
$$x_2=x_{2,0},\; x_{2,1},\; x_{2,2},\ldots , x_{2,l-1},\; x_{2,l}$$
in $\langle z_2\rangle$ such that $o(x_{2,k})=p_2^{\beta_2 +k}$ and $x_{2,k}^{p_2}=x_{2,k-1}$ for $1\leq k\leq l$. We have $\langle x_{2,l}\rangle=\langle z_2\rangle$, as $o(x_{2,l})=p_2^{\beta_2 +l}=p_2^{\eta_2}=o(z_2)$. The fact that $p_2$ and $o(x_1)$ are relatively prime implies
$\left\langle x_{1}x_{2,k-1}\right\rangle=\left\langle x_{1}x_{2,k}^{p_2}\right\rangle=\left\langle (x_{1}x_{2,k})^{p_2}\right\rangle$. So
$$\deg(x_{1}x_{2,k-1})=\deg((x_{1}x_{2,k})^{p_2})\geq \deg(x_{1}x_{2,k})$$
for $1\leq k\leq l$, where the last inequality holds by Proposition \ref{prop.abelian} as $Z(P_2)$ is noncyclic and $2\phi(p_1)\geq p_1$. We thus have
$$\deg(x)=\deg(x_1x_2)=\deg(x_1x_{2,0})\geq \deg(x_1x_{2,1})\geq \cdots\geq \deg(x_1x_{2,l})=\deg(x_1z_2)>\deg(z_2),$$
where the second last equality holds as $\langle x_1x_{2,l}\rangle =\langle x_1z_2\rangle$.
This completes the proof.

\section*{Acknowledgement}
The first author would like to thank the National Institute of Science
Education and Research, Bhubaneswar, for the facilities provided when working as a visiting
fellow in the School of Mathematical Sciences.

\noindent\underline{\bf Addresses}

\begin{enumerate}
	
	\item[] {\bf Ramesh Prasad Panda ({\tt rppanda@iitk.ac.in})}
	\begin{enumerate}[\rm(1)]
		\item Department of Mathematics and Statistics, Indian Institute of Technology Kanpur, Uttar Pradesh--208016, India.
	\end{enumerate}
	
\item[] {\bf Kamal Lochan Patra ({\tt klpatra@niser.ac.in})\\
Binod Kumar Sahoo ({\tt bksahoo@niser.ac.in})}
\begin{enumerate}[\rm(1)]
\item School of Mathematical Sciences, National Institute of Science Education and Research, Bhubaneswar, P.O.-Jatni, District-Khurda, Odisha-752050, India.

\item Homi Bhabha National Institute, Training School Complex, Anushakti Nagar, Mumbai-400094, India.
\end{enumerate}

\end{enumerate}
\end{document}